\documentclass{amsart}%
\usepackage{amssymb}
\usepackage{amsmath}
\usepackage{latexsym}
\usepackage{hyperref}
\usepackage{amsfonts}
\usepackage{graphicx}%
\newtheorem{theorem}{Theorem}

\newcommand{\be}{\begin{equation}}
\newcommand{\ee}{\end{equation}}
\newcommand{\bea}{\begin{eqnarray}}
\newcommand{\eea}{\end{eqnarray}}

\begin{document}
\title{A lower bound for the scalar curvature of noncompact nonflat Ricci shrinkers}
\author{Bennett Chow}
\address{Department of Mathematics, University of California San Diego, La Jolla, CA 92093}
\email{{benchow@math.ucsd.edu}}
\author{Peng Lu}
\address{Department of Mathematics, University of Oregon, Eugene, OR 97403}
\email{{penglu@uoregon.edu}}
\author{Bo Yang}
\address{Department of Mathematics, University of California San Diego, La Jolla, CA 92093}
\email{{b5yang@math.ucsd.edu}}

\begin{abstract}
We show that recent work of Ni and Wilking \cite{Ni} yields the result that a
noncompact nonflat Ricci shrinker has at most quadratic scalar curvature
decay. The examples of noncompact K\"{a}hler--Ricci shrinkers by Feldman,
Ilmanen, and Knopf \cite{F-I-K} exhibit that this result is sharp.

\end{abstract}
\maketitle


Let $\left(  \mathcal{M}^{n},g,f\right)  $ be a complete shrinking gradient
Ricci soliton (\emph{Ricci shrinker} for short) with $R_{ij}+\nabla_{i}%
\nabla_{j}f-\frac{1}{2}g_{ij}=0$ and $R+|\nabla f|^{2}-f=0$. Bing-Long Chen
\cite{ChenB} proved that $R\geq0$. If $\left(  \mathcal{M},g\right)  $ is not
isometric to Euclidean space, then $R>0$ (see Stefano Pigola, Michele Rimoldi,
and Alberto Setti \cite{P-R-S} and Shijin Zhang \cite{Zhang}).

Recently, Lei Ni and Burkhard Wilking \cite{Ni} proved that on any noncompact
nonflat Ricci shrinker and for any $\epsilon>0$, there exists a constant
$C_{\epsilon}>0$ such that $R(x)\geqslant C_{\epsilon}d(x,O)^{-2-\epsilon}$
wherever $d(x,O)$ is sufficiently large. The purpose of this note is to
observe the following version of their result.

\begin{theorem}
Let $(\mathcal{M}^{n},g,f)$ be a complete noncompact nonflat shrinking
gradient Ricci soliton. Then for any given point $O\in\mathcal{M}$ there
exists a constant $C_{0}>0$ such that $R(x)d(x,O)^{2}\geqslant C_{0}^{-1}$
wherever $d(x,O)\geqslant C_{0}$. Consequently, the asymptotic scalar
curvature ratio of $g$ is positive.
\end{theorem}

\begin{proof}
Recall that Huai-Dong Cao and De-Tang Zhou \cite{Cao-Zhou} proved that there
exists a positive constant $C_{1}$ such that $f$ satisfies the estimate:%
\begin{equation}
\frac{1}{4}\left[  \left(  d(x,O)-C_{1}\right)  _{+}\right]  ^{2}\leq
f(x)\leq\frac{1}{4}\left(  d(x,O)+2f{(O)}^{\frac{1}{2}}\right)  ^{2}%
,\label{Sharp lower bound for f}%
\end{equation}
where $c_{+}\doteqdot\max(c,0)$ (see also Fu-Quan Fang, Jian-Wen Man, and
Zhen-Lei Zhang \cite{FangManZHang} and, for an improvement, Robert Haslhofer
and Reto M\"{u}ller \cite{H-M}). Define the $f$-Laplacian $\Delta_{f}%
\doteqdot\Delta-\nabla f\cdot\nabla$. We have $0<R+\left\vert \nabla
f\right\vert ^{2}=f=\frac{n}{2}-\Delta_{f}f$. Recall that (see \cite{ELM} for
example)%
\begin{equation}
\Delta_{f}R=-2\left\vert \operatorname{Rc}\right\vert ^{2}+R.\label{test3}%
\end{equation}

Note that%
\begin{align}
\Delta_{f}\left(  f^{-1}\right)   &  =f^{-1}-f^{-2}\left(  \frac{n}{2}%
-2\frac{\left\vert \nabla f\right\vert ^{2}}{f}\right)  ,\label{test1}\\
\Delta_{f}\left(  f^{-2}\right)   &  =2f^{-2}-f^{-3}\left(  n-6\frac
{\left\vert \nabla f\right\vert ^{2}}{f}\right)  .\label{test2}%
\end{align}
Using (\ref{test3}) and (\ref{test1}), we compute for any $c>0$%
\begin{equation}
\Delta_{f}\left(  R-cf^{-1}\right)  \leqslant R-cf^{-1}+cf^{-2}\left(
\frac{n}{2}-2\frac{\left\vert \nabla f\right\vert ^{2}}{f}\right)  .
\end{equation}
Define $\phi\doteqdot R-cf^{-1}-cnf^{-2}$. By (\ref{test2}) we obtain%
\begin{equation}
\Delta_{f}\phi\leqslant\phi-cnf^{-3}\left(  \frac{f}{2}-n\right)
-cf^{-4}\left(  2f+6n\right)  \left\vert \nabla f\right\vert ^{2}%
.\label{maximum principle1}%
\end{equation}

Choosing $c>0$ sufficiently small, we have $\phi>0$ inside $B(O,C_{1}+3n)$,
where $C_{1}$ is as in (\ref{Sharp lower bound for f}). If $\inf
_{\mathcal{M}-B(O,C_{1}+3n)}\phi\doteqdot-\delta<0$, then by
(\ref{Sharp lower bound for f}) there exists $\rho>C_{1}+3n$ such that
$\phi>-\frac{\delta}{2}$ in $\mathcal{M}-B\left(  O,\rho\right)  $. Thus a
negative minimum of $\phi$ is attained at some point $x_{0}$ outside of
$B(O,C_{1}+3n)$. By the maximum principle, evaluating
(\ref{maximum principle1}) at $x_{0}$ yields $\frac{f\left(  x_{0}\right)
}{2}-n\leq0$. However, (\ref{Sharp lower bound for f}) implies that
$f(x_{0})\geqslant\frac{9n^{2}}{4}$, a contradiction. We conclude that $R\geq
cf^{-1}+cnf^{-2}$ on $\mathcal{M}$. The theorem follows from
(\ref{Sharp lower bound for f}).
\end{proof}

\textbf{Remark}. Mikhail Feldman, Tom Ilmanen, and Dan Knopf
\cite{F-I-K} constructed complete noncompact K\"{a}hler--Ricci
shrinkers on the total spaces of $k$-th powers of tautological line
bundles over the complex projective space
$\mathbb{C}\mathbb{P}^{n-1}$ for $0<k<n$. These examples, which have
Euclidean volume growth and quadratic scalar curvature decay, show
that Theorem 1 is sharp.\smallskip

\textbf{Acknowledgment.} We would like to thank Lei Ni for informing us of his
result with Burkhard Wilking.

\end{document}